\documentclass[12pt]{amsart}
\usepackage{amsfonts}
\usepackage{amsmath}
\usepackage{amssymb}
\usepackage{graphicx}
\usepackage{amscd}

\newtheorem{theorem}{Theorem}

\newtheorem{proposition}[theorem]{Proposition}

\begin{document}
	\begin{abstract}
		We prove $L^{p}$-boundedness of oscillating multipliers on some classes of rank
		one locally symmetric spaces.
	\end{abstract}
	
	\title[Oscillating multipliers ]{Oscillating multipliers on rank one locally
		symmetric spaces}
	\author{Effie Papageorgiou}
	\email{papageoe@math.auth.gr}
	\curraddr{Department of Mathematics, Aristotle University of Thessaloniki,
		Thessaloniki 54.124, Greece }
	\maketitle
	
	\section{Introduction and statement of the results}
	
	Oscillating multipliers on $\mathbb{R}^{n}$ are bounded functions of the
	type
	\begin{equation*}
	m_{\alpha ,\beta }(\xi )=\left\Vert \xi \right\Vert ^{-\beta }e^{i\left\Vert
		\xi \right\Vert ^{\alpha }}\theta \left( \xi \right) ,
	\end{equation*}%
	where $\alpha ,\beta >0$ and $\theta \in C_{0}^{\infty }\left( \mathbb{R}%
	\right) $ which vanishes near zero, and equals to $1$ outside the ball $%
	B\left( 0,2\right) $. Let $T_{\alpha ,\beta }$ be the operator which in the
	Fourier transform variables is given by
	\begin{equation*}
	\widehat{(T_{\alpha ,\beta }f)}\left( \xi \right) =m_{\alpha ,\beta }(\xi )%
	\hat{f}\left( \xi \right) \text{, \ }f\in C_{0}^{\infty }\left( \mathbb{R}%
	^{n}\right) ,
	\end{equation*}%
	i.e. $T_{\alpha ,\beta }$ is a convolution operator with kernel the inverse
	Fourier transform of $m_{\alpha ,\beta }$. The $L^{p}$-boundedness of $%
	T_{\alpha ,\beta }$ on $\mathbb{R}^{n}$ is extensively studied. See for
	example \cite{HIR,WAI,STE,FE,SCHO} for $\alpha \in \left( 0,1\right) $ and
	\cite{PERAL} for $\alpha =1$.
	
	The $L^{p}$-boundedness of oscillating multipliers has been studied also in
	various geometric contexts as Riemannian manifolds, Lie groups and symmetric
	spaces. See for example \cite{GIUME,ALEX,MAR,GEORG, LO} and the references
	therein.
	
	In the present work we deal with oscillating multipliers on rank one locally
	symmetric spaces. To state our results, we need to introduce some notation
	(for details see Section 2). Let $G$ be a semi-simple, non-compact,
	connected Lie group with finite center and let $K$ be a maximal compact
	subgroup of $G$. We consider the symmetric space of non-compact type $X=G/K$%
	. Let $G=KAN$ be the Iwasawa decomposition of $G$. If $A\cong \mathbb{R},$
	we say that $X$ has rank one. Recall that rank one symmetric spaces are the
	real, complex, and quaternionic hyperbolic spaces, denoted $\mathbb{H}^{n}(%
	\mathbb{R})$, $\mathbb{H}^{n}(\mathbb{C})$ and $\mathbb{H}^{n}(\mathbb{H})$,
	$n\geq 2$, and the octonionic hyperbolic plane $\mathbb{H}^{2}(\mathbb{O})$.
	Throughout this paper we shall assume that $\dim X=n\geq 2$ and rank$X=1$.
	
	In \cite{GIUME}, Giulini and Meda consider the multiplier
	\begin{equation*}
	\ m_{\alpha,
		\beta}(\lambda)=(\lambda^2+\rho^2)^{-\beta/2}e^{i(\lambda^2+\rho^2)^{
			\alpha/2}}, \alpha>0, \operatorname{Re}{\beta} \geq 0, \lambda >0,
	\end{equation*}
	where $\rho $ is the half sum of positive roots, counted with their
	multiplicities. As in the case of $\mathbb{R}^n$, we denote by $T_{\alpha,
		\beta}$ the convolution operator with kernel $\kappa_{\alpha ,\beta }$, the
	inverse spherical Fourier transform of $m_{\alpha ,\beta}$ in the sense of
	distributions. Then,
	\begin{equation}
	T_{\alpha ,\beta }(f)(x)=\int_{G}\kappa _{\alpha ,\beta
	}(xy^{-1})f(y)dy,\quad f\in C_{0}^{\infty }(X).  \label{operatorX1}
	\end{equation}
	Note that
	\begin{equation}
	T_{\alpha ,\beta }=\Delta _{X}^{-\beta /2}e^{i\Delta _{X}^{\alpha /2}},\
	\alpha >0,\ {\operatorname{Re}}{\beta }\geq 0,  \label{om1}
	\end{equation}
	where $\Delta_{X}$ is the Laplace-Beltrami operator on $X$. In \cite{GIUME},
	the $L^{p}$- boundedness of $T_{\alpha ,\beta }$ is investigated on rank one
	symmetric spaces and the following theorem is proved:
	
	\begin{theorem}[Giulini, Meda]
		\label{GM}If $p\in (1,\infty )$, then
		
		(i) If $\alpha <1$, then $T_{\alpha ,\beta }$ is bounded on $L^{p}(X)$,
		provided that $\operatorname{Re}{\beta }>\alpha n\left\vert 1/p-1/2\right\vert $,
		
		(ii) if $\alpha =1$, then $T_{\alpha ,\beta }$ is bounded on $L^{p}(X)$,
		provided that $\operatorname{Re}{\beta }>\left( n-1\right) \left\vert
		1/p-1/2\right\vert $,
		
		(iii) if $\alpha >1$, then $T_{\alpha ,\beta }$ is bounded on $L^{p}(X)$ if
		and only if $p=2$.
	\end{theorem}
	
	As it is noticed in \cite{GIUME}, the results above for the case $\alpha
	\leq 1$ are less precise than in the Euclidean case, since it is not known
	what happens at the critical indices $\operatorname{Re}{\beta }=\alpha n\left\vert
	1/p-1/2\right\vert $ for $\alpha \in \left( 0,1\right) $ and $\operatorname{Re}{%
		\beta }=\left( n-1\right) \left\vert 1/p-1/2\right\vert $ for $\alpha =1$.
	
	Let us now present the case of locally symmetric spaces. Let $\Gamma $ be a
	discrete and torsion free subgroup of $G$ and let us consider the locally
	symmetric space $M=\Gamma \backslash X=\Gamma \backslash G/K.$ Then $M$,
	equipped with the projection of the canonical Riemannian structure of $X$,
	becomes a Riemannian manifold.
	
	To define oscillating multipliers on $M$, we first observe that if $f\in
	C_{0}^{\infty }(M)$, then, the function $T_{\alpha ,\beta }f$ defined by (%
	\ref{operatorX1}), is right $K$-invariant and left $\Gamma $-invariant. So $%
	T_{\alpha ,\beta }$ can be considered as an operator acting on functions on $%
	M$, which we shall denote by $\widehat{T}_{\alpha ,\beta }$. Note that the
	Laplace-Beltrami operator $\Delta _{M}$ is the projection of $\Delta _{X}$.
	So from (\ref{om1}), it follows that
	\begin{equation}
	\widehat{T}_{\alpha ,\beta }f=\Delta _{M}^{-\beta /2}e^{i\Delta _{M}^{\alpha
			/2}}f,\quad f\in C_{0}^{\infty }(M).  \label{om2}
	\end{equation}%
	In this paper we deal with the $L^{p}$-boundedness of $\widehat{T}_{\alpha
		,\beta }$ and we prove the analogue of Theorem \ref{GM}. We treat the cases $%
	\alpha \in \left( 0,1\right) $, $\alpha =1$ and $\alpha >1$, separately.
	
	\bigskip
	
	\textbf{Case 1}\textit{. }$\alpha \in \left( 0,1\right) $. The main
	ingredient for the proof of our results is the Kunze and Stein phenomenon on
	locally symmetric spaces, proved in \cite{LOMAjga}, and which states that
	there exist $\eta _{\Gamma }\in \mathfrak{a}^{\ast }$ and $s\left(
	p\right)>0 $, $p\in \left( 1,\infty \right) $, such that for every $K$-bi-invariant function $\kappa $, the convolution operator $\ast \left\vert
	\kappa \right\vert $ with kernel $\left\vert \kappa \right\vert $ satisfies
	the estimate
	\begin{equation}
	\left\Vert \ast \left\vert \kappa \right\vert \right\Vert
	_{L^{p}(M)\rightarrow L^{p}(M)}\leq \int_{G}\left\vert \kappa \left(
	g\right) \right\vert \varphi _{-i\eta _{\Gamma }}\left( g\right) ^{s\left(
		p\right) }dg,  \label{ks00}
	\end{equation}
	where $\varphi _{\lambda }$ is the spherical function with index $\lambda $.
	For more details see Section 2.
	
	We say that $M=\Gamma \backslash X$ belongs in the class $(KS)$ if the Kunze
	and Stein phenomenon is valid on it. The class $(KS)$ is described in detail
	in \cite[Section 1]{LOMAjga}. We note that $M\in (KS)$ for all discrete
	groups $\Gamma $, if $X=\mathbb{H}^{n}(\mathbb{H}),\mathbb{H}^{2}(\mathbb{O}%
	) $, while if $X=\mathbb{H}^{n}(\mathbb{R}),\mathbb{H}^{n}(\mathbb{C})$,
	then $M\in (KS)$, provided that $\Gamma $ is amenable \cite{FOMAMA, LOMAjga}.
	
	We have the following theorem:
	
	\begin{theorem}
		\label{alphaless1}If $M$ belongs in the class $(KS)$, then for $\alpha \in
		\left( 0,1\right) $, $\widehat{T}_{\alpha ,\beta }$ is bounded on $L^{p}(M)$%
		, provided that $\operatorname{Re}{\beta >}\alpha n|1/p-1/2|.$
	\end{theorem}
	
	\textbf{Case 2}\textit{. }$\alpha =1$. This case is of particular interest
	since by (\ref{om2}), $\widehat{T}_{1,\beta }=\Delta _{M}^{-\beta
		/2}e^{i\Delta _{M}^{1/2}}$ and thus $\widehat{T}_{1,\beta }$ is related to
	the wave operator.
	
	The $L^{p}$-boundedness of the operator $T_{1,\beta }$ is investigated in
	\cite{PERAL} for the case of $\mathbb{R}^{n}$, in \cite{LO} in a very
	general geometric context including Riemannian manifolds of bounded geometry
	and Lie groups of polynomial or exponential growth, and in \cite{GIUME} for
	rank one symmetric spaces.
	
	Denote by $\delta (\Gamma )$ the critical exponent of the group $\Gamma $:
	\begin{equation}
	\delta (\Gamma )=\inf \left\{ s>0:P_{s}(x,y)<+\infty \right\} ,
	\label{criticalexp}
	\end{equation}%
	where
	\begin{equation}
	P_{s}(x,y)=\sum_{\gamma \in \Gamma }e^{-sd(x,\gamma y)}  \label{poinc}
	\end{equation}%
	is the Poincar\'{e} series. Note that $\delta (\Gamma )\in \lbrack 0,2\rho ]$%
	.
	
	We say that $\Gamma $ is of convergence type according as the series $%
	P_{\delta (\Gamma )}(x,y)$ converges and of divergence type otherwise.
	
	We say that $\Gamma $ belongs in the class $(CT)$ if $\delta (\Gamma )=2\rho
	$ and $\Gamma $ is of convergence type. For example, if $\Gamma \subset SO(n,1)$, 
	then $\Gamma \in (CT)$ if it is of the second kind, i.e. the limit set
	$\Lambda (\Gamma )$ is not equal to the whole of 
	$\partial \mathbb{H}^{n}(\mathbb{R})$, \cite[Theorem 1.6.2]{NI}
	
	In \cite{RO}, Roblin gives
	a criterion of convergence for the case of rank one locally symmetric
	spaces. Recall that a limit point $\xi$ belong in the conical limit set $\Lambda _{c}\left( \Gamma \right) $
	if there is some geodesic ray $\beta_{t } $ tending to $\xi $  as $t\rightarrow \infty $
	and a constant $c>0$ such that the orbit $\Gamma x_{0}$, $x_{0}\in X$, accumulates to $\xi $
	within the closed $c$-neighborhood of $\beta_{t} $. In \cite[Theorem 1.7]{RO}
	it is proved that $\Gamma$ is of convergence type  iff $%
	\mu _{y_{0}}\left( \Lambda _{c}\left( \Gamma \right) \right) =0$, where
	$\mu _{y_{0}}$, $y_{0} \in X$ is a $\Gamma$-invariant Patterson-Sullivan density.
	For more details for the class $(CT)$, see Section 2.
	
	\begin{theorem}
		\label{alphaeq1} If either $\delta (\Gamma )<2\rho $ or $\Gamma \in (CT)$,
		then $\widehat{T}_{1,\beta }$ is bounded on $L^{p}(M)$, $p\in (1,\infty )$,
		provided that $\operatorname{Re}{\beta >}(n-1)|1/p-1/2|$.
	\end{theorem}
	
	\textbf{Case 3}\textit{. }$\alpha >1$. As it is shown in Section 3,
	\begin{equation*}
	\widehat{T}_{\alpha ,\beta }=\frac{1}{\Gamma \left( \beta /\alpha \right) }%
	\int_{0}^{\infty }\sigma ^{\left( \beta /\alpha \right) -1}e^{\left(
		i-\sigma \right) \Delta _{M}^{\alpha /2}}d\sigma .
	\end{equation*}%
	But, by the spectral theorem
	\begin{equation*}
	\left\Vert e^{\left( i-\sigma \right) \Delta _{M}^{\alpha /2}}\right\Vert
	_{L^{2}(M)\rightarrow L^{2}(M)}\leq e^{-\sigma \rho ^{\alpha }}.
	\end{equation*}%
	So,
	\begin{equation*}
	\left\Vert \widehat{T}_{\alpha ,\beta }\right\Vert _{L^{2}(M)\rightarrow
		L^{2}(M)}\leq c,
	\end{equation*}%
	and this is the only we can say for the $L^{p}$-boundedness of $\widehat{T}%
	_{\alpha ,\beta }$ if $\alpha >1$. On the contrary, in \cite{GIUME} Giulini
	and Meda observe that the multiplier $m_{\alpha ,\beta }$ is not bounded on
	any strip $S_{\varepsilon }=\left\{ |\operatorname{Im}\lambda |<\varepsilon \rho
	\right\} $, $\varepsilon \leq 1$, and consequently by the necessary part of
	the multiplier theorem of Clerc and Stein \cite{CLERC}, they conclude that $%
	T_{\alpha ,\beta }$ is bounded on $L^{p}(X)$ \textit{iff} $p=2$.
	
	The paper is organized as follows. In Section 2 we present the necessary
	tools we need for our proofs, and in Section 3 we give the proofs of our
	results.
	
	\section{Preliminaries}
	
	In this section we recall some basic facts about symmetric spaces and
	locally symmetric spaces we will use for the proof of our results. For
	details see \cite{AN, HEL, FOMAMA, LOMAjga}.
	
	Let $G$ be a semisimple Lie group, connected, noncompact, with finite center
	and let $K$ be a maximal compact subgroup of $G$ and consider the symmetric
	space $X=G/K$. In the sequel we assume that $\dim X=n$. Denote by $\mathfrak{%
		g}$ and $\mathfrak{k}$ the Lie algebras of $G$ and $K$. Let also $\mathfrak{p%
	}$ be the subspace of $\mathfrak{g}$ which is orthogonal to $\mathfrak{k}$
	with respect to the Killing form. The Killing form induces a $K$-invariant
	scalar product on $\mathfrak{p}$ and hence a $G$-invariant metric on $G/K$.
	Denote by $\Delta _{X}$ the Laplace-Beltrami operator on $X$, by $d(.,.)$
	the Riemannian distance and by $dx$ the associated measure on $X$. Let $%
	\Gamma $ be a discrete torsion free subgroup of $G$. Then, the locally
	symmetric space $M=\Gamma \backslash X$, equipped with the projection of the
	canonical Riemannian structure of $X$, becomes a Riemannian manifold. We
	denote by $\Delta _{M}$ the Laplacian on $M$.
	
	Fix $\mathfrak{a}$ a maximal abelian subspace of $\mathfrak{p}$ and denote
	by $\mathfrak{a}^{\ast }$ the real dual of $\mathfrak{a}$. Let $A$ be the
	analytic subgroup of $G$ with Lie algebra $\mathfrak{a}.$ Let $\mathfrak{a}
	^{+}\subset \mathfrak{a}$ and let $\overline{\mathfrak{a}^{+}}$ be its
	closure. Put $A^{+}=\exp \mathfrak{a}^{+}$. Its closure in $G$ is $\overline{
		A^{+}}=\exp \overline{\mathfrak{a}^{+}}$. We have the Cartan decomposition
	\begin{equation*}
	\ G=K(\overline{A^{+}})K=K(\exp \overline{\mathfrak{a}^{+}})K.
	\end{equation*}
	We say that $a$ is a root vector if for every $H\in \mathfrak{a}$
	\begin{equation*}
	\ [H,X]=a(H)X,\;X\in \mathfrak{g}.
	\end{equation*}
	Denote by $\rho $ the half sum of positive roots, counted with their
	multiplicities. Denote by $x_{0}=eK$ the origin of $X$. If $x,y\in X$, then
	there are isometries $g,h\in G$ such that $x=gx_{0}$ and $y=hx_{0}$. Then,
	\begin{equation}
	d(x,y)=d(gx_{0},hx_{0})=d(x_{0},g^{-1}hx_{0}).  \label{d(x,y)}
	\end{equation}
	But, by the Cartan decomposition
	\begin{equation}
	g^{-1}h=k\left( \exp H(g^{-1}h)\right) k^{\prime },\quad k,k^{\prime }\in
	K,\;H(g^{-1}h)\in \overline{\mathfrak{a}_{+}}.  \label{kak}
	\end{equation}
	In the rank one case, we have that
	\begin{equation*}
	\ A^{+}=\left\{ \exp H:H\in \overline{\mathfrak{a}_{+}}\right\} =\left\{
	\exp tH_{0}:t>0\right\} ,
	\end{equation*}
	where $H_{0}\in \overline{\mathfrak{a}_{+}}$ with $\Vert H_{0}\Vert =1$.
	From (\ref{d(x,y)}) and (\ref{kak}) it follows that
	\begin{equation}
	d(x,y)=\Vert H\Vert =t\Vert H_{0}\Vert =t.  \label{tgamma}
	\end{equation}
	
	\subsection{The spherical Fourier transform and the Kunze and Stein
		phenomenon}
	
	Denote by $S(K\backslash G/K)$ the Schwartz space of $K$-bi-invariant
	functions on $G$. The spherical Fourier transform $\mathcal{H}$ is defined
	by
	\begin{equation*}
	\ \mathcal{H}f(\lambda )=\int_{G}f(x)\varphi _{\lambda }(x)\;dx,\quad
	\lambda \in \mathfrak{a^{\ast }},\quad f\in S(K\backslash G/K),
	\end{equation*}
	where $\varphi _{\lambda }$ are the elementary spherical functions on $G$.
	
	Let $S(\mathfrak{a^{\ast }})$ be the usual Schwartz space on $\mathfrak{\
		a^{\ast }}$, and let us denote by $S(\mathfrak{a^{\ast }})^{W}$ the subspace
	of $W$-invariants in $S(\mathfrak{a^{\ast }})$, where $W$ is the Weyl group
	associated to the root system of $(\mathfrak{g},\mathfrak{a})$. Then, by a
	celebrated theorem of Harish-Chandra, $\mathcal{H}$ is an isomorphism
	between $S(K\backslash G/K)$ and $S(\mathfrak{a^{\ast }})$ and its inverse
	is given by
	\begin{equation*}
	\ (\mathcal{H}^{-1}f)(x)=c\int_{\mathfrak{a\ast }}f(\lambda )\varphi
	_{-\lambda }(x)\frac{d\lambda }{|\mathbf{c}(\lambda )|^{2}},\quad x\in
	G,\quad f\in S(\mathfrak{a^{\ast }})^{W},
	\end{equation*}%
	where $\mathbf{c}(\lambda )$ is the Harish-Chandra function.
	
	In \cite{LOMAjga} it is proved the following analogue \cite{CO} of Kunze and
	Stein phenomenon, for convolution operators on a class of locally symmetric
	spaces. Let $C_{\rho }$ be the convex body in $\mathfrak{a^{\ast }}$
	generated by the vectors $\left\{ w\rho ;\;w\in W\right\} .$ Let also $%
	\lambda _{0}$ be the bottom of the $L^{2}$-spectrum of $\Delta _{M}$. Then
	there exists a vector $\eta _{\Gamma }\in C_{\rho }\cap S(0,(\rho
	^{2}-\lambda _{0})^{1/2})$, where $S(0,r)$ is the Euclidean sphere of $%
	\mathfrak{a^{\ast }}$, such that for all $p\in (1,\infty )$ and for every $K$-bi-invariant function $\kappa $, the convolution operator $\ast |\kappa |$
	with kernel $|\kappa |$ satisfies the estimate
	\begin{equation}
	\Vert \ast |\kappa |\Vert _{L^{p}(M)\rightarrow L^{p}(M)}\leq
	\int_{G}|\kappa (g)|\varphi _{-i\eta _{\Gamma }}(g)^{s(p)}dg,  \label{Kunze}
	\end{equation}%
	where
	\begin{equation*}
	\ s(p)=2\min ((1/p),(1/p^{\prime })),
	\end{equation*}%
	and $p^{\prime }$ is the conjugate of $p$. For more details, see \cite%
	{FOMAMA, LOMAjga}.
	
	\subsection{The class $(CT)$}
	
	In \cite{RO}, Roblin investigates the question of convergence or divergence
	type of a discrete group $\Gamma $ of isometries of $CAT(-1)$ spaces. Note
	that $CAT(-1)$ spaces contain all rank one symmetric spaces. To state the
	results of Roblin, we need to introduce some notation.
	
	Denote by $\mu _{x}$, $x\in X$, a $\Gamma $-invariant Patterson-Sullivan
	density: that is a family of finite and mutually absolutely continuous
	measures on $\partial X$ satisfying the following conditions:
	
	(i) for any $x,y\in X$,
	\begin{equation*}
	\frac{d\mu _{y}}{d\mu _{x}}\left( \xi \right) =e^{-\delta \left( \Gamma
		\right) \beta _{\xi }\left( x,y\right) },\text{ }\xi \in \partial X.
	\end{equation*}%
	Here $\beta _{\xi }\left( x,y\right) $ is the Busemann function:
	\begin{equation*}
	\beta _{\xi }\left( x,y\right) =\lim_{t\rightarrow \infty }\left( d\left(
	\xi _{t},x\right) -d\left( \xi _{t},y\right) \right) ,
	\end{equation*}%
	where $\xi _{t}$ is a geodesic ray tending to $\xi $ as $t\rightarrow \infty
	$.
	
	(ii) for any $\gamma \in \Gamma $ and $x\in X$, $\gamma ^{\ast }\mu _{x}=\mu
	_{\gamma x}$.
	
	Note that for any $x\in X$, $\mu _{x}$ is supported on the limit set $%
	\Lambda \left( \Gamma \right) $. Note also that the celebrated Patterson-Sullivan
	construction insures the existence of such conformal densities
	in various geometric contexts as
	Hadamard manifolds and $CAT\left( -1\right)$ spaces, \cite{COO,YUE}.

	If $\mu _{x_{0}}$ is normalized to be a probability measure, then in \cite[%
	Theorem 1.7]{RO}, Roblin proves that $\Gamma $ is of convergence type iff $%
	\mu _{x_{0}}\left( \Lambda _{c}\left( \Gamma \right) \right) =0$.
	
	As far as it concerns divergence type groups, the question is investigated
	in \cite{SU,COR} for rank one symmetric spaces and in \cite{DAL,LEU} for
	rank greater than $2$. More precisely in \cite[Proposition 2]{SU} Sullivan
	treats the case of the real hyperbolic space and shows that $\Gamma\subset SO(n,1)$ is of
	divergence type if it is geometrically finite and convex co-compact, i.e. $%
	\Gamma \backslash C\left( \Lambda \left( \Gamma \right) \right) $ is
	compact, where $C\left( \Lambda \left( \Gamma \right) \right) $ is the
	convex hull of the limit set of $\Gamma $. In \cite[Proposition 3.7]{COR}
	Corlette and Iozzi treat the case of all rank one symmetric spaces and
	improve the result of \cite{SU} in proving that if $\Gamma $ is
	geometrically finite subgroup of isometries of a rank one symmetric space,
	then $\Gamma $ is of divergence type.
	
	In \cite{LEU} Leuzinger investigates the case when $G$ possesses Kazhdan's
	property $(T)$, i.e. when $G$ has no simple factors locally isomorphic to $%
	SO(n,1)$ or $SU(n,1)$, \cite[Ch. 2]{HAVA}. For example, the rank one
	symmetric spaces which possess property $(T)$ are $H^{n}\left( \mathbb{H}%
	\right) =Sp(n,1)/Sp(n)$, and $H^{2}\left( \mathbb{O}\right)
	=F_{4}^{-20}/Spin\left( 9\right) $. In \cite[Main Theorem]{LEU} it is proved
	that if $G$ is as above, then the following are equivalent:
	
	(i) $M$ is a lattice, i.e. Vol($M$)$<\infty $,
	
	(ii) $\delta (\Gamma )=2\rho $,
	
	(iii) $\Gamma $ is of divergence type.
	
	It follows then that if $\Gamma\in Sp\left(n,1\right)$ (resp. $\Gamma\in F_{4}^{20})$
	and $\Gamma \backslash \mathbb{H}^{n}(\mathbb{H%
	})$  (resp. $\Gamma \backslash \mathbb{H}^{2}(\mathbb{O})$)  is a lattice, then
	$\Gamma$ is of divergence type.
	So, only discrete subgroups $\Gamma $ of $SO(n,1)$ or $SU(n,1)$ with $\delta (\Gamma )=2\rho $
	can possibly be in the class $\left(CT\right)$.
	
	\section{Proof of the results}
	
	\subsection{Proof of Theorem \protect\ref{alphaless1}}
	
	We start by performing a decomposition of the kernel $\kappa _{\alpha ,\beta
	}=\mathcal{H}^{-1}(m_{\alpha ,\beta }).$ We write
	\begin{equation*}
	\ \kappa _{\alpha ,\beta }=\kappa _{\alpha ,\beta }^{0}+\kappa _{\alpha
		,\beta }^{\infty },
	\end{equation*}%
	with $\kappa _{\alpha ,\beta }^{0},\kappa _{\alpha ,\beta }^{\infty }$, $K$%
	-bi-invariant functions that satisfy $supp(\kappa _{\alpha ,\beta
	}^{0})\subset B(0,2)$ and $supp(k_{\alpha ,\beta }^{\infty })\subset
	B(0,1)^{c}.$ Denote by $\widehat{T}_{\alpha ,\beta }^{0}$ and $\widehat{T}%
	_{\alpha ,\beta }^{\infty }$ the corresponding convolution operators with
	kernels $\kappa _{\alpha ,\beta }^{0}$ and $\kappa _{\alpha ,\beta }^{\infty
	}$.
	
	The operator $\widehat{T}_{\alpha ,\beta }^{0}$ is the \textquotedblleft
	local\textquotedblright\ part of $\widehat{T}_{\alpha ,\beta }$ and has the
	same behavior as its analogue in the Euclidean context.
	
	\begin{proposition}
		\label{localprop} \label{local part}For $\alpha \in \left( 0,1\right) $, $%
		\widehat{T}_{\alpha ,\beta }^{0}$ is bounded on $L^{p}(M)$, provided that $%
		\operatorname{Re}{\beta >}\alpha n|1/p-1/2|$.
	\end{proposition}
	
	\begin{proof}
		The continuity of $\widehat{T}_{\alpha ,\beta }^{0}$ on $L^p(M)$ follows
		from \cite[Proposition 13]{LOMAanna}. Indeed, observe first that $\widehat{T}%
		_{\alpha ,\beta }^{0}$ can be defined as an operator on the group $G$, and
		then, apply the local result of \cite{GIUME}, to conclude its boundedness on
		$L^{p}(G)$, for all $p\in (1,\infty )$ such that $\operatorname{Re} {\beta >}\alpha
		n|1/p-1/2|$. The continuity of $\widehat{T}_{\alpha ,\beta }^{0}$ on $%
		L^{p}(M)$, follows by applying Herz's Theorem A, \cite{HE2}.
	\end{proof}
	
	Note that Proposition \ref{localprop} is valid for all discrete and torsion
	free subgroups of $G$.
	
	To finish the proof of Theorem \ref{alphaless1} it remains to prove the $%
	L^{p}$ boundedness of $\widehat{T}_{\alpha ,\beta }^{\infty }$. Here we
	shall need the assumption that the Kunze and Stein phenomenon is valid on $M$%
	.
	
	\begin{proposition}
		\label{infinity}If $M$ belongs in the class $(KS)$, and $\alpha \in \left(
		0,1\right) $, then $\widehat{T}_{\alpha ,\beta }^{\infty }$ is bounded on $%
		L^{p}(M)$ for all $p\in \left( 1,\infty \right) $.
	\end{proposition}
	
	For the proof of the proposition above we shall make use of the Kunze and
	Stein phenomenon. For that we need to introduce some notation.
	
	For $p\in (1,\infty )$, set
	\begin{equation}
	v_{\Gamma }(p)=2\min \{(1/p),(1/p^{\prime })\}\frac{|\eta _{\Gamma }|}{\rho }%
	+|(2/p)-1|,  \label{v(p)}
	\end{equation}%
	where $p^{\prime }$ is the conjugate of $p$ and $\eta _{\Gamma }\in
	\mathfrak{a^{\ast }}$ is the vector appearing in (\ref{Kunze}).
	
	For $N\in \mathbb{N}$, $v\in \mathbb{R}$ and $\theta \in \left( 0,1\right) $, we say that the multiplier $m$ belongs in the class $\mathcal{M}
	(v,N,\theta )$, if
	
	\begin{itemize}
		\item $m$ is analytic inside the strip $\mathcal{T}^v=\mathfrak{a^*}
		+ivC_{\rho}$ and
		
		\item for all $k\in \mathbb{N}$ with $k\leq N$, $\partial ^{k}m(\lambda )$
		extends continuously to the whole of $\mathcal{T}^{v}$ with
		\begin{equation}
		|\partial ^{k}m(\lambda )|\leq c(1+|\lambda |^{2})^{-k\theta /2}:=<\lambda
		>^{-k\theta }.  \label{decayrate}
		\end{equation}
	\end{itemize}
	
	We have the following:
	
	\begin{proposition}
		\label{propos1} Fix $p\in (1,\infty )$ and consider an even function $m\in
		\mathcal{M}(v,N,\theta )$, with $v>v_{\Gamma }(p)$ and $N=\left[ \frac{n+1}{
			2\theta }\right] +1$. If $\kappa $ is the inverse spherical transform of $m$
		in the sense of distributions, and $\kappa ^{\infty }$ its part away from
		the origin, then, the convolution operator $\widehat{T}_{\kappa }^{\infty }$
		with kernel $\kappa ^{\infty }$, is bounded on $L^{p}(M).$
	\end{proposition}
	
	\begin{proof}
		Since $M$ belongs in the class $(KS)$, then, according to Kunze and Stein
		phenomenon, we have that
		\begin{eqnarray}
		||\widehat{T}_{\kappa }^{\infty }||_{L^{p}(M)\rightarrow L^{p}(M)} &\leq
		&\int_{G}\left\vert \kappa ^{\infty }(g)\right\vert \varphi _{-i\eta
			_{\Gamma }}(g)^{s(p)}dg  \notag  \label{SkLp} \\
		&\leq &\int_{|g|\geq 1}\left\vert \kappa (g)\right\vert \varphi _{-i\eta
			_{\Gamma }}(g)^{s(p)}dg:=I.  \label{KS}
		\end{eqnarray}
		
		So, to finish the proof of Proposition \ref{propos1}, we have to show that
		if $m$ satisfies (\ref{decayrate}), then the integral $I$ in (\ref{SkLp}) is
		finite. To prove that, we will slightly modify the proof of the Theorem 1 in
		\cite{LOMAjga} which is based on Proposition 5 of \cite{AN}. To estimate the
		integral $I$, we shall estimate $\kappa ^{\infty }$ over concentric shells.
		Set
		
		\begin{equation*}
		V_{r}=\{H\in \mathfrak{a}:|H|\leq r\},
		\end{equation*}
		and
		\begin{equation*}
		V_{r}^{+}=V_{r}\cap \overline{\mathfrak{a}_{+}}.
		\end{equation*}
		Note that in the rank one case, $V_{r}=\left[ -r,r\right] $, and $V_{r}^{+}= %
		\left[ 0,r\right] $.
		
		Set also
		\begin{equation*}
		\ U_{r}=\{x=k_{1}(\exp H)k_{2}\in G:k_{1},k_{2}\in K,\;H\in
		V_{r}^{+}\}=K(\exp V_{r}^{+})K.
		\end{equation*}
		Using the Cartan decomposition of $G$, the integral $I$ is written as
		\begin{equation}
		I=\sum_{j\geq 1}\int_{U_{j+1}\backslash U_{j}}|\kappa (g)|\varphi _{-i\eta
			_{\Gamma }}(g)^{s(p)}dg:=\sum_{j\geq 1}I_{j}.  \label{cocentric}
		\end{equation}
		Set $b=n-1$ and let $b^{\prime }$ be the smallest integer $\geq b/2$. Set
		also $N=\left[ \frac{n+1}{2\theta }\right] +1$, $\theta \in (0,1)$. In \cite[
		p.645]{LOMAjga}, using \cite[p.608]{AN}, it is proved that if $m\in \mathcal{%
			\ M}\left( v,N,\theta \right) $, then for $v>v_{\Gamma }(p)$ and $j\geq 1$,
		\begin{equation}
		I_{j}\leq cj^{-N}\sum_{0\leq k\leq N}\left( \int_{\mathfrak{a\ast }
		}(<\lambda >^{b^{\prime }-N+k}|\partial _{\lambda }^{k}m(\lambda +i\rho
		)|)^{2}d\lambda \right) ^{1/2}.  \label{Ij}
		\end{equation}
		
		Then, using the estimates of the derivatives of $m(\lambda )$
		\begin{equation*}
		|\partial ^{k}m(\lambda )|\leq c<\lambda >^{-k\theta },
		\end{equation*}
		we obtain that
		\begin{align*}
		I_{j}& \leq cj^{-N}\sum_{0\leq k\leq N}\left( \int_{\mathfrak{a\ast }
		}(<\lambda >^{b^{\prime }-N+k}<\lambda >^{-k\theta })^{2}d\lambda \right)
		^{1/2} \\
		& \leq cj^{-N}\sum_{0\leq k\leq N}\left( \int_{\mathfrak{a\ast }}(<\lambda
		>^{b^{\prime }-N+k(1-\theta )})^{2}d\lambda \right) ^{1/2} \\
		& \leq cj^{-N}\left( \int_{\mathfrak{a\ast }}<\lambda >^{2(b^{\prime
			}-\theta N)}d\lambda \right) ^{1/2} \\
		& \leq cj^{-N}\left( \int_{0}^{\infty }(1+\lambda ^{2})^{b^{\prime }-\theta
			\left( \lbrack \frac{n+1}{2\theta }]+1\right) }d\lambda \right) ^{1/2}.
		\end{align*}
		Using the fact that $b^{\prime }\leq n/2$ and that $[\frac{n+1}{2\theta }]=
		\frac{n+1}{2\theta }-q$, $q\in \lbrack 0,1)$, it follows that
		\begin{align}
		I_{j}& \leq cj^{-N}\left( \int_{0}^{\infty }(1+\lambda ^{2})^{\frac{n}{2}
			-\theta \left( \frac{n+1}{2\theta }-q\right) -\theta }d\lambda \right) ^{1/2}
		\label{j^-N} \\
		& \leq cj^{-N}\left( \int_{0}^{\infty }(1+\lambda ^{2})^{-\frac{1}{2}-\theta
			(1-q)}d\lambda \right) ^{1/2}  \notag \\
		& \leq cj^{-N},  \notag
		\end{align}
		since $2\left( \frac{1}{2}+\theta (1-q)\right) >1$.
		
		From (\ref{j^-N}) it follows that
		\begin{equation*}
		\ \int_{|g|\geq 1}\left\vert \kappa ^{\infty }(g)\right\vert \varphi
		_{-i\eta _{\Gamma }}(g)^{s(p)}dg=\sum_{j\geq 1}I_{j}\leq c\sum_{j\geq
			1}j^{-N}<\infty ,
		\end{equation*}
		since $N>1$.
	\end{proof}
	
	\bigskip
	
	\begin{proof}[Proof of Proposition \protect\ref{infinity}]
		Note first that $m_{\alpha ,\beta }\left( \lambda \right) $ has poles only
		at $\lambda =\pm i\rho $. So, the function $\lambda \longrightarrow
		m_{\alpha ,\beta }\left( \lambda \right) $ is analytic in the strip $S_{\rho
		}=\left\{ z\in \mathbb{C}:\left\vert \operatorname{Im}z\right\vert <\rho \right\} $.
		Secondly, for every $k\in \mathbb{N}$, it holds that
		\begin{equation}
		\left\vert \partial ^{k}m_{a,\beta }(\lambda )\right\vert \leq c(1+\lambda
		)^{-Re\beta -k(1-a)}\leq c(1+\lambda ^{2})^{-k(1-\alpha )/2},\text{ }\lambda
		\in S_{\rho }.  \label{in2}
		\end{equation}
		
		Note also that for any $p\in \left( 1,\infty \right) $, from (\ref{v(p)}) it
		follows that $v_{\Gamma }(p)<1$. So, for every $p\in \left( 1,\infty \right)
		$, there exists $v^{\prime }\left( p\right) <1$, such that $v^{\prime
		}\left( p\right) >v_{\Gamma }(p)$. It follows that $\mathfrak{a^{\ast }}%
		+iv^{\prime }\left( p\right) C_{\rho }\subset S_{\rho }$, which combined
		with (\ref{in2}), implies that for any $p\in \left( 1,\infty \right) $, and $%
		N\in \mathbb{N}$, $m_{\alpha ,\beta }\in \mathcal{M}(v^{\prime }\left(
		p\right) ,N,1-\alpha )$. Thus, Proposition \ref{propos1} applies and
		Proposition \ref{infinity} follows.
	\end{proof}
	
	\subsection{Proof of Theorem \protect\ref{alphaeq1}}
	
	To begin with, let us recall that in \cite{SCHO}, the operator $T_{\alpha
		,\beta }$, $Re\beta \geq 0$, $\alpha >0$ is also expressed by the integral:
	\begin{equation}
	T_{\alpha ,\beta }(f)(x)=\frac{1}{\Gamma (\beta /\alpha )}\int_{0}^{\infty
	}\sigma ^{\beta /\alpha -1}(f\ast q_{\sigma ,\alpha })(x)\;d\sigma ,\quad
	f\in C_{0}^{\infty }(X),  \label{Ta,b}
	\end{equation}%
	where $q_{\sigma ,\alpha }$, $\sigma >0$, is the inverse spherical Fourier
	transform of the function $\lambda \rightarrow e^{(i-\sigma )(\lambda
		^{2}+\rho ^{2})^{\alpha /2}}$. Note that $q_{\sigma ,\alpha }$ is $K$-bi-invariant as the inverse spherical Fourier transform of an even
	function, \cite{AN}. Then, observe that the convolution operator $%
	T_{q_{\sigma ,\alpha }}=\ast q_{\sigma ,\alpha }$ is equal to $e^{(i-\sigma
		)\Delta _{X}^{\alpha /2}}$. So, by the spectral theorem, $T_{q_{\sigma
			,\alpha }}$ is bounded on $L^{2}(X)$, with
	\begin{equation}
	\Vert T_{q_{\sigma }}\Vert _{L^{2}(X)\rightarrow L^{2}(X)}\leq \sup_{\lambda
		>0}\left\vert e^{(i-\sigma )(\lambda ^{2}+\rho ^{2})^{\alpha /2}}\right\vert
	\leq e^{-\sigma \rho ^{\alpha }}.  \label{L(2(X))}
	\end{equation}
	
	For simplicity, for $\alpha =1$, set $q_{\sigma ,1}=q_{\sigma }$. Thus, $%
	T_{q_{\sigma }}=\ast q_{\sigma }$ is equal to $e^{(i-\sigma )\Delta
		_{X}^{1/2}}$. Let us now define the operator $\widehat{T}_{1,\beta }$ on the
	quotient $M=\Gamma \backslash X$. For that, recall that $\widehat{T}%
	_{q_{\sigma }}$, as it is the case of $\widehat{T}_{1,\beta }$, is initially
	defined as a convolution operator
	\begin{equation}
	\widehat{T}_{q_{\sigma }}f(x)=\int_{G}q_{\sigma }(y^{-1}x)f(y)dy,\quad f\in
	C_{0}^{\infty }(M).  \label{33}
	\end{equation}
	
	Set $q_{\sigma }(x,y)=q_{\sigma }(y^{-1}x)$ and
	\begin{equation}
	\widehat{q_{\sigma }}(x,y)=\sum_{\gamma \in \Gamma }q_{\sigma }(x,\gamma
	y)=\sum_{\gamma \in \Gamma }q_{\sigma }((\gamma y)^{-1}x).  \label{conv}
	\end{equation}
	
	\begin{proposition}
		\label{forallgroups}If either $\delta (\Gamma )<2\rho $ or $\Gamma \in (CT) $%
		, then the series (\ref{conv}) is convergent and the operator $\widehat{T}
		_{q_{\sigma }}$ on $M$ is given by
		\begin{equation}
		\widehat{T}_{q_{\sigma }}f(x)=\int_{M}\hat{q}_{\sigma }(x,y)f(y)dy,\quad
		f\in C_{0}^{\infty }(M).  \label{Tqhat}
		\end{equation}
	\end{proposition}
	
	\begin{proof}
		By the Cartan decomposition of $G$, we may write
		\begin{equation*}
		(\gamma y)^{-1}x=k_{\gamma }\exp (t_{\gamma }H_{0})k_{\gamma }^{\prime },
		\end{equation*}
		where $t_{\gamma }>0,\;k_{\gamma },k_{\gamma }^{\prime }\in K,$ and $%
		\;H_{0}\in \mathfrak{a}_{+}$ with $\left\Vert H_{0}\right\Vert =1$. Then,
		since $q_{\sigma }$ is $K$-bi-invariant, we get that $q_{\sigma }((\gamma
		y)^{-1}x)=q_{\sigma }(\exp t_{\gamma }H_{0}).$
		
		Recall that in \cite[p.103]{GIUME} it is proved that
		\begin{equation}
		\left\vert q_{\sigma }(\exp tH_{0})\right\vert \leq
		\begin{cases}
		c\sigma (t+1)^{-3/2}e^{-2\rho t}, & \sigma >1,t>0, \\
		c(t+1)^{-3/2}e^{-2\rho t}, & 0<\sigma \leq 1,t>2.%
		\end{cases}
		\label{q1}
		\end{equation}%
		Then,
		\begin{eqnarray*}
			|\widehat{q}_{\sigma }(\exp tH_{0})| &\leq &\sum_{\gamma \in \Gamma
			}|q_{\sigma }(\exp t_{\gamma }H_{0})| \\
			&\leq &\sum_{\left\{ \gamma \in \Gamma :\;t_{\gamma }\leq 2\right\}
			}q_{\sigma }(\exp t_{\gamma }H_{0})+\sum_{\left\{ \gamma \in \Gamma
				:\;t_{\gamma }>2\right\} }q_{\sigma }(\exp t_{\gamma }H_{0}).
		\end{eqnarray*}%
		Note that the first sum above is finite, since $\Gamma $ is discrete. Thus,
		we shall deal only with the second sum. Using (\ref{tgamma}) and the
		estimates (\ref{q1}) of $q_{\sigma }$, we have
		\begin{eqnarray*}
			\sum_{\left\{ \gamma \in \Gamma :\;t_{\gamma }>2\right\} }q_{\sigma }(\exp
			t_{\gamma }H_{0}) &\leq &\sum_{\left\{ \gamma \in \Gamma :\;t_{\gamma
				}>2\right\} }c_{\sigma }(t_{\gamma }+1)^{-3/2}e^{-2\rho t_{\gamma }} \\
			&\leq &c_{\sigma }\sum_{\left\{ \gamma \in \Gamma :d(x,\gamma y)>2\right\}
			}e^{-2\rho d(x,\gamma y)} \\
			&\leq &c_{\sigma }P_{2\rho }\left( x,y\right)
		\end{eqnarray*}%
		which is convergent, since by our assumption, either $\delta (\Gamma )<2\rho
		$ or $\Gamma \in (CT)$.
		
		To prove (\ref{Tqhat}), we follow \cite[Prop.4]{FOMAMA}. Since $q_{\sigma }$
		and $f$ are right $K$-invariant, from (\ref{33}) we get that
		\begin{equation*}
		\widehat{T}_{q_{\sigma }}(f)(x)=\int_{G}q_{\sigma
		}(xy^{-1})f(y)dy=\int_{X}q_{\sigma }(x,y)f(y)dy.
		\end{equation*}
		Further, since $f$ is left $\Gamma $-invariant, by Weyl's formula we obtain
		that
		\begin{align*}
		\widehat{T}_{q_{\sigma }}(f)(x) &=\int_{X}q_{\sigma
		}(x,y)f(y)dy=\int_{\Gamma \backslash X}\left( \sum_{\gamma \in \Gamma
		}q_{\sigma }(x,\gamma y)f(\gamma y)\right) dy \\
		&=\int_{M}\widehat{q}_{\sigma }(x,y)f(y)dy.
		\end{align*}
	\end{proof}
	
	\begin{proposition}
		\label{interpol}If either $\delta (\Gamma )<2\rho $ or $\Gamma \in (CT)$,
		then for all $p\in (1,\infty )$, there are constants $c_{p},k_{p}>0$ such
		that
		\begin{equation}
		\Vert \widehat{T}_{q_{\sigma }}\Vert _{L^{p}(M)\rightarrow L^{p}(M)}\leq
		c_{p}%
		\begin{cases}
		e^{-k_{p}\sigma }, & \sigma \geq 1, \\
		\sigma ^{(1-n)(1/2-1/p)}, & \sigma <1.%
		\end{cases}
		\label{pp}
		\end{equation}
	\end{proposition}
	
	\begin{proof}
		Recall that $\widehat{T}_{q_{\sigma }}=e^{(i-\sigma )\Delta_M ^{1/2}}$.
		Thus, by the spectral theorem,
		\begin{equation}  \label{L2M}
		\Vert \widehat{T}_{q_{\sigma }}\Vert _{L^{2}(M)\rightarrow L^{2}(M)} \leq
		\sup_{\lambda >0}\left\vert e^{-(i-\sigma )(\lambda ^{2}+\rho
			^{2})^{1/2}}\right\vert \leq e^{-\sigma \rho }.
		\end{equation}
		
		Next, note that if $f\in C_{0}^{\infty }(M)$, then $\Vert f\Vert _{L^{\infty
			}(M)}=\Vert f\Vert _{L^{\infty }(X)}$, since $f$ is left $\Gamma $-invariant. Then, again by Weyl's formula it follows that
		\begin{eqnarray}
		|\widehat{T}_{q_{\sigma }}f(x)| &=&\left\vert \int_{M}\hat{q}_{\sigma
		}(x,y)f(y)dy\right\vert =\left\vert \int_{\Gamma \backslash X}\sum_{\gamma
			\in \Gamma }{q_{\sigma }}(x,\gamma y)f(y)dy\right\vert  \notag  \label{24} \\
		&=&\left\vert \int_{X}q_{\sigma }(x,y)f(y)dy\right\vert \leq \Vert f\Vert
		_{L^{\infty }(X)}\Vert q_{\sigma }\Vert _{L^{1}(X)}  \label{++} \\
		&=&\Vert f\Vert _{L^{\infty }(M)}\Vert q_{\sigma }\Vert _{L^{1}(X)}.  \notag
		\end{eqnarray}%
		But, in \cite[p.101]{GIUME} it is proved that
		\begin{equation}
		\Vert {q_{\sigma }}\Vert _{L^{1}(X)}\leq c%
		\begin{cases}
		\sigma , & \sigma \geq 1, \\
		\sigma ^{(1-n)/2}, & \sigma <1.%
		\end{cases}
		\label{q_sigmaest}
		\end{equation}%
		From (\ref{++}) and (\ref{q_sigmaest}), it follows that $\widehat{T}%
		_{q_{\sigma }}$ is bounded on $L^{\infty }(M)$, with
		\begin{equation}
		\Vert \widehat{T}_{q_{\sigma }}\Vert _{L^{\infty }(M)\rightarrow L^{\infty
			}(M)}\leq \Vert q_{\sigma }\Vert _{L^{1}(X)}.  \label{Loo}
		\end{equation}
		
		By interpolation between (\ref{L2M}) and (\ref{Loo}) and duality, we deduce
		the boundedness of $\widehat{T}_{q_{\sigma }}$ on $L^{p}(M)$, $p>1$.
		
		Further, by the Riesz-Thorin interpolation theorem we have
		\begin{equation*}
		\Vert \widehat{T}_{q_{\sigma }}\Vert _{L^{p_{\theta }}(M)\rightarrow
			L^{p_{\theta }}(M)}\leq c\Vert \widehat{T}_{q_{\sigma }}\Vert _{L^{{2}%
			}(M)\rightarrow L^{{2}}(M)}^{1-\theta }\Vert \widehat{T}_{q_{\sigma }}\Vert
		_{L^{{\infty }}(M)\rightarrow L^{{\infty }}(M)}^{\theta },
		\end{equation*}%
		with $\frac{1}{p_{\theta }}=\frac{1-\theta }{2}$. Choosing $\theta
		=1-(2/p),\;p>2$, we have that
		\begin{equation*}
		\Vert \widehat{T}_{q_{\sigma }}\Vert _{L^{p}(M)\rightarrow L^{p}(M)}\leq
		c\Vert \widehat{T}_{q_{\sigma }}\Vert _{L^{p}(M)\rightarrow
			L^{p}(M)}^{2/p}\Vert \widehat{T}_{q_{\sigma }}\Vert _{L^{\infty
			}(M)\rightarrow L^{\infty }(M)}^{1-(2/p)},
		\end{equation*}%
		and from (\ref{L2M}), (\ref{q_sigmaest}) and (\ref{Loo}), it is
		straightforward to obtain (\ref{pp})\ for $p>2$. The claim for $p\in \left(
		1,2\right) $ follows by duality.
	\end{proof}
	
	\begin{proof}[\textit{End of proof of Theorem \protect\ref{alphaeq1}}.]
		Using (\ref{Ta,b}), we have that
		\begin{eqnarray*}
			\Vert \widehat{T}_{1,\beta }\Vert _{L^{p}(M)\rightarrow L^{p}(M)} &\leq &%
			\frac{1}{\Gamma (\beta )}\int_{0}^{\infty }\sigma ^{\beta -1}\Vert \widehat{T%
			}_{q_{\sigma }}\Vert _{L^{p}(M)\rightarrow L^{p}(M)}\;d\sigma \\
			&\leq &\frac{1}{\Gamma (\beta )}\int_{0}^{1}\sigma ^{\beta -1}\Vert \widehat{%
				T}_{q_{\sigma }}\Vert _{L^{p}(M)\rightarrow L^{p}(M)}\;d\sigma \\
			&+&\frac{1}{\Gamma (\beta )}\int_{1}^{\infty }\sigma ^{\beta -1}\Vert
			\widehat{T}_{q_{\sigma }}\Vert _{L^{p}(M)\rightarrow L^{p}(M)}\;d\sigma .
		\end{eqnarray*}%
		Applying the estimates (\ref{pp}), we get:
		\begin{eqnarray*}
			\Vert \widehat{T}_{1,\beta }\Vert _{L^{p}(M)\rightarrow L^{p}(M)} &\leq &%
			\frac{1}{\Gamma (\beta )}\int_{0}^{1}\sigma ^{\beta -1}\sigma
			^{(1-n)(1/2-1/p)}\;d\sigma \\
			&+&\frac{1}{\Gamma (\beta )}\int_{1}^{\infty }\sigma ^{\beta
				-1}e^{-k_{p}\sigma }\;d\sigma .
		\end{eqnarray*}%
		The second integral above is finite, while the first is convergent provided
		that
		\begin{equation*}
		\ \operatorname{Re}\beta -1-(n-1)(1/2-1/p)>-1,\text{ or }\operatorname{Re}\beta
		/(n-1)>1/2-1/p,
		\end{equation*}%
		and the claim follows for $p>2$. The case $p<2$ follows by duality.
	\end{proof}
	
	\textbf{Acknowledgement}. The author would like to thank Professor Michel
	Marias for his generous and valuable help, and for bringing the problem to
	our attention, as well as Professor Anestis Fotiadis for stimulating
	discussions and support.

\end{document}